\documentclass[runningheads]{llncs}

\usepackage[utf8]{inputenc}
\usepackage{amsmath,amssymb}
\usepackage{mathtools}
\usepackage{microtype}
\usepackage{mathrsfs}
\usepackage{url}
\usepackage{enumerate}

\usepackage{tikz}
\usetikzlibrary{positioning,arrows,decorations.pathmorphing}
\tikzset{node distance=15mm, auto}
\let\temp\phi
\let\phi\varphi
\let\varphi\temp

\let\temp\epsilon
\let\epsilon\varepsilon
\let\varepsilon\temp

\let\temp\theta
\let\theta\vartheta
\let\vartheta\temp

\newif\ifdraft
\drafttrue

\ifdraft
  \usepackage{color}
  \usepackage{soul}

\else

\fi

\newcommand{\ie}{\emph{i.e.}}
\newcommand{\eg}{\emph{e.g.}}

\DeclareMathOperator{\id}{id}
\DeclareMathOperator{\fix}{fix}
\DeclareMathOperator{\pfix}{pfix}

\DeclareMathOperator{\op}{op}
\DeclareMathOperator{\Tr}{Tr}

\newcommand{\C}{\ensuremath{\mathscr{C}}}

\newcommand{\D}{\ensuremath{\mathscr{D}}}

\newcommand{\V}{\ensuremath{\mathscr{V}}}

\newcommand{\Poset}{\ensuremath{\mathbf{Poset}}}
\newcommand{\DStoch}{\ensuremath{\mathbf{DStoch}_{\le 1}}}
\newcommand{\Set}{\ensuremath{\mathbf{Set}}}
\newcommand{\DCPO}{\ensuremath{\mathbf{DCPO}}}
\newcommand{\CPO}{\ensuremath{\mathbf{CPO}}}
\newcommand{\DcpoOp}{\ensuremath{\mathbf{DcpoOp}}}
\newcommand{\Rel}{\ensuremath{\mathbf{Rel}}}
\newcommand{\PInj}{\ensuremath{\mathbf{PInj}}}
\newcommand{\FHilb}{\ensuremath{\mathbf{FHilb}}}
\newcommand{\Vect}{\ensuremath{\mathbf{Vect}}}
\newcommand{\CPS}{\ensuremath{\mathbf{CP}^*}}

\newcommand{\DagCat}{\ensuremath{\mathbf{DagCat}}}

\newcommand{\To}{\Rightarrow}
\newcommand{\tot}{\xrightarrow}
\newcommand{\Tot}{\xRightarrow}

\newcommand{\iso}{\cong}
\newcommand{\dle}{\sqsubseteq}

\renewcommand*{\dag}{\ensuremath{\dagger}}

\newcommand{\inv}{\overline}
\newcommand{\iinv}[1]{\overline{\overline{#1}}}

\begin{document}

\title{Inversion, Iteration, and the\\ Art of Dual Wielding\thanks{The
author would like to thank Martti Karvonen, Mathys Rennela, and Robert Glück
for their useful comments, corrections, and suggestions; and to acknowledge the
support given by \emph{COST Action IC1405 Reversible computation: Extending
horizons of computing.}}}
\author{Robin Kaarsgaard\orcidID{0000-0002-7672-799X}}
\authorrunning{R. Kaarsgaard}

\institute{DIKU, Department of Computer Science, University of Copenhagen \\
\email{robin@di.ku.dk}}

\maketitle

\begin{abstract}
The humble $\dagger$ (``dagger'') is used to denote two different operations in
category theory: Taking the \emph{adjoint} of a morphism (in dagger categories)
and finding the \emph{least fixed point} of a functional (in categories
enriched in domains). While these two operations are usually considered
separately from one another, the emergence of reversible notions of computation
shows the need to consider how the two ought to interact.

In the present paper, we wield both of these daggers at once and consider
dagger categories enriched in domains. We develop a notion of a monotone dagger
structure as a dagger structure that is well behaved with respect to the
enrichment, and show that such a structure leads to pleasant inversion
properties of the fixed points that arise as a result. Notably, such a
structure guarantees the existence of \emph{fixed point adjoints}, which we
show are intimately related to the \emph{conjugates} arising from a canonical
involutive monoidal structure in the enrichment. Finally, we relate the results
to applications in the design and semantics of reversible programming
languages.

\end{abstract}

\section{Introduction} 
\label{sec:introduction}
Dagger categories are categories in which each morphism $X \tot{f} Y$ can be
assigned an \emph{adjoint} $Y \tot{f^\dagger} X$ subject to certain equations.
In recent years, dagger categories have been used to capture aspects of
\emph{inversion} in both reversible~\cite{James2012,James2014,Kaarsgaard2017}
and quantum~\cite{AbramskyCoecke2004,Selinger2007,Coecke2016} computing.
Likewise, domain theory and categories enriched in domains (see, \eg,
\cite{AbramskyJung1994,Esik2009,Esik2015,Adamek1995,Barr1992,SmythPlotkin1982})
have been successful since their inception in modelling both recursive
functions and data types in programming via \emph{fixed points}.

A motivating example of the interaction between adjoints and fixed points is
found in the reversible functional programming language
Rfun~\cite{Yokoyama2012}, as the interaction between program inversion and
recursion. In this language, inverses of recursive functions can be constructed
in a particularly straightforward way, namely as recursive functions with
function body the inverse of the function body of the original function.
Previously, the author and others showed that this phenomenon appears in
\emph{join inverse categories}, a particular class of domain-enriched dagger
categories suitable for modelling classical reversible computing, as
\emph{fixed point adjoints}~\cite{Kaarsgaard2017} to the functionals (\ie,
second-order continuous functions) used to model recursive functions.

Several questions remain about these fixed point adjoints, however. Notably:
Are these fixed point adjoints canonical? Why do they arise in classical
reversible computing, and do they arise elsewhere as well? To answer these
questions requires us to develop the art of wielding the two daggers offered by
dagger categories and domain-enriched categories at once. We argue that
well-behaved interaction between the dagger and domain-enrichments occurs when
the dagger is locally monotone, \ie, when $f \dle g$ implies $f^\dagger \dle
g^\dagger$. We show that the functionals on \C{} form an involutive monoidal
category, which also proves surprisingly fruitful in unifying seemingly
disparate concepts from the literature under the banner of \emph{conjugation of
functionals}. Notably, we show that the conjugate functionals arising from this
involutive structure coincide with fixed point adjoints~\cite{Kaarsgaard2017},
and that they occur naturally both in proving the ambidexterity of dagger
adjunctions~\cite{HeunenKarvonen2016} and in natural transformations that
preserve the dagger (including dagger traces~\cite{Selinger2011}).

While these results could be applied to model a reversible functional
programming language with general recursion and parametrized functions (such as
an extended version of Theseus~\cite{James2014}), they are general enough to
account for even certain probabilistic and nondeterministic models of
computation, such as the category \Rel{} of sets and relations, and the
category \DStoch{} of finite sets and subnormalized doubly stochastic maps.

\emph{Overview:} A brief introduction to the relevant background material on
dagger categories, (\DCPO-)enriched categories, iteration categories, and
involutive monoidal categories is given in Section~\ref{sec:background}. In
Section~\ref{sec:dcpo_dagger_categories} the concept of a \emph{monotone dagger
structure} on a \DCPO-category is introduced, and it is demonstrated that such
a structure leads to the existence of fixed point adjoints for (ordinary and
externally parametrized) fixed points, given by their conjugates. We also
explore natural transformations in this setting, and develop a notion of
\emph{self-conjugate} natural transformations, of which $\dagger$-trace
operators are examples. Finally, we discuss potential applications and avenues
for future research in Section~\ref{sec:applications}, and end with a few
concluding remarks in Section~\ref{sec:conclusion}.

\section{Background} 
\label{sec:background}
Though familiarity with basic category theory, including monoidal categories, is
assumed, we recall here some basic concepts relating to dagger categories,
(\DCPO)-enriched categories, iteration categories, and involutive monoidal
categories~\cite{Jacobs2012,Beggs2009}. The material is only covered here
briefly, but can be found in much more detail in the numerous texts on dagger
category theory (see, \eg, \cite{Selinger2007,AbramskyCoecke2004,Heunen2009,Karvonen2019}), enriched category theory (for which \cite{Kelly1982} is the standard text), and domain
theory and iteration categories (see, \eg, \cite{AbramskyJung1994,Esik2015}).

\subsection{Dagger categories} 
\label{sub:dagger_categories}
A dagger category (or \dag-category) is a category equipped with a suitable
method for flipping the direction of morphisms, by assigning to each morphism
an \emph{adjoint} in a manner consistent with composition. They are formally
defined as follows.

\begin{definition}
A dagger category is a category $\C$ equipped with an functor $(-)^\dag
: \C^{\op} \to \C$ satisfying that $\id_X^\dagger = \id_X$ and
$f^{\dagger\dagger} = f$ for all identities $X \tot{\id_X} X$ and morphisms $X
\tot{f} Y$.
\end{definition}

Dagger categories, dagger functors (\ie, functors $F$ satisfying $F(f^\dag) =
F(f)^\dag$), and natural transformations form a 2-category, \DagCat.

A given category may have several different daggers which need not agree. An
example of this is the groupoid of finite-dimensional Hilbert spaces and linear
isomorphisms, which has (at least!) two daggers: One maps linear isomorphisms
to their linear inverse, the other maps linear isomorphisms to their hermitian
conjugate. The two only agree on the unitaries, \ie, the linear isomorphisms
which additionally preserve the inner product. For this reason, one would in
principle need to specify \emph{which} dagger one is talking about on a given
category, though this is often left implicit (as will also be done here).

Let us recall the definition of the some interesting properties of morphisms in
a dagger category: By theft of terminology from linear algebra, say that a
morphism $X \tot{f} X$ in a dagger category is \emph{hermitian} or
\emph{self-adjoint} if $f = f^\dag$, and \emph{unitary} if it is an isomorphism
and $f^{-1} = f^\dag$. Whereas objects are usually considered equivalent if
they are isomorphic, the ``way of the
dagger''~\cite{HeunenKarvonen2016,Karvonen2019} dictates that all structure in
sight must cooperate with the dagger; as such, objects ought to be considered
equivalent in dagger categories only if they are isomorphic via a unitary map.

We end with a few examples of dagger categories. As discussed above, \FHilb{}
is an example (\emph{the} motivating one, even~\cite{Selinger2007}) of dagger
categories, with the dagger given by hermitian conjugation. The category
\PInj{} of sets and partial injective functions is a dagger category (indeed,
it is an \emph{inverse category}~\cite{Kastl1979,Cockett2002}) with $f^\dag$
given by the partial inverse of $f$. Similarly, the category \Rel{} of sets and
relations has a dagger given by $R^\dag = R^{\circ}$, \ie, the relational
converse of $R$. Noting that a dagger subcategory is given by the existence of
a faithful dagger functor, it can be shown that \PInj{} is a dagger subcategory
of \Rel{} with the given dagger structures.

\subsection{\DCPO-categories and other enriched categories} 
\label{sub:_dcpo_categories_and_other_enriched_categories}
Enriched categories (see, \eg, \cite{Kelly1982}) capture the idea that homsets
on certain categories can (indeed, ought to) be understood as something other
than sets -- or in other words, as objects of another category than \Set. A
category \C{} is \emph{enriched} in a monoidal category \V{} if all homsets
$\C(X,Y)$ of \C{} are objects of \V{}, and for all objects $X,Y,Z$ of \C{},
\V{} has families of morphisms $\C(Y,Z) \otimes \C(X,Y) \to \C(X,Z)$ and $I \to
\C(X,X)$ corresponding to composition and identities in \C, subject to
commutativity of diagrams corresponding to the usual requirements of
associativity of composition, and of left and right identity. As is common, we
will often use the shorthand ``\C{} is a \V-category'' to mean that \C{} is
enriched in the category \V.

We focus here on categories enriched in the category of \emph{domains} (see,
\eg, \cite{AbramskyJung1994}), \ie, the category \DCPO{} of pointed directed
complete partial orders and continuous maps. A partially ordered $(X, \dle)$ is
said to be directed complete if every directed set (\ie, a \emph{non-empty} $A
\subseteq X$ satisfying that any pair of elements of $A$ has a supremum in $A$) has a supremum in $X$. A function $f$ between directed
complete partial orders is monotone if $x \dle y$ implies $f(x) \dle f(y)$ for
all $x,y$, and continuous if $f(\sup A) = \sup_{a \in A} \{f(a)\}$ for each
directed set $A$ (note that continuity implies monotony). A directed complete
partial order is \emph{pointed} if it has a least element $\bot$ (or, in other
words, if also the empty set has a supremum), and a function $f$
between such is called \emph{strict} if $f(\bot) = \bot$ (\ie, if also the
supremum of the empty set is preserved\footnote{This is \emph{not} the
case in general, as continuous functions are only required to preserve least
upper bounds of directed sets, which, by definition, does not include the empty
set.}). Pointed directed complete partial orders and continuous maps form a
category, \DCPO.

As such, a category enriched in \DCPO{} is a category \C{} in which homsets
$\C(X,Y)$ are directed complete partial orders, and composition is continuous.
Additionally, we will require that composition is strict (meaning that $\bot
\circ f = \bot$ and $g \circ \bot = \bot$ for all suitable morphisms $f$ and
$g$), so that the category is actually enriched in the category \DCPO! of
directed complete partial orders and strict continuous functions, though we
will not otherwise require functions to be strict.

Enrichment in \DCPO{} provides a method for constructing morphisms in the
enriched category as least fixed points of continuous functions between
homsets: This is commonly used to model recursion. Given a
continuous function $\C(X,Y) \tot{\phi} \C(X,Y)$, by Kleene's fixed point
theorem there exists a least fixed point $X \tot{\fix \phi} Y$ given by
$\sup_{n \in \omega}\{\phi^n(\bot)\}$, where $\phi^n$ is the $n$-fold
composition of $\phi$ with itself.

\subsection{Parametrized fixed points and iteration categories} 
\label{sub:iteration_categories}
Related to the fixed point operator is the \emph{parametrized fixed point
operator}, an operator $\pfix$ assigning morphisms of the form $X \times Y
\tot{\psi} X$ to a morphism $Y \tot{\pfix \psi} X$ satisfying equations such as
the \emph{parametrized fixed point identity}
$$
\pfix \psi = \psi \circ \langle \pfix \psi, \id_Y \rangle
$$
and others (see, \eg, \cite{Hyland2008,Esik2009}). Parametrized fixed points
are used to solve domain equations of the form $x = \psi(x,p)$ for some given
parameter $p \in Y$. Indeed, if for a continuous function $X \times Y
\tot{\psi} X$ we define $\psi^0(x,p) = x$ and $\psi^{n+1}(x,p) =
\psi(\psi^n(x,p),p)$, we can construct its parametrized fixed point in \DCPO{} in a way reminiscent of the usual fixed point by
$$
  (\pfix \psi)(p) = \sup_{n \in \omega} \{\psi^n(\bot_X, p) \} \enspace.
$$
In fact, a parametrized fixed point operator may be derived from an ordinary
fixed point operator by $(\pfix \psi)(p) = \fix \psi(-,p)$. Similarly, we may
derive an ordinary fixed point operator from a parametrized one by considering
a morphism $X \tot{\phi} X$ to be parametrized by the terminal object $1$, so
that the fixed point of $X \tot{\phi} X$ is given by the parametrized fixed
point of $X \times 1 \tot{\pi_1} X \tot{\phi} X$.

The parametrized fixed point operation is sometimes also called a \emph{dagger
operation}~\cite{Esik2009}, and denoted by $f^\dag$ rather than $\pfix f$.
Though this is indeed the other dagger that we are wielding, we will use the
phrase ``parametrized fixed point'' and notation ``$\pfix$'' to avoid
unnecessary confusion.

An \emph{iteration category}~\cite{Esik2015} is a cartesian category with a
parametrized fixed point operator that behaves in a canonical way. The
definition of an iteration category came out of the observation that the
parametrized fixed point operator in a host of concrete categories (notably
\DCPO{}) satisfy the same identities. This lead to an elegant semantic
characterization of iteration categories, due to \cite{Esik2015}.

\begin{definition}
  An \emph{iteration category} is a cartesian category with a parametrized 
  fixed point operator satisfying all identities (of the parametrized fixed 
  point operator) that hold in \DCPO.
\end{definition}

Note that the original definition defined iteration categories in relation to
the category $\mathbf{CPO}_m$ of $\omega$-complete partial orders and monotone
functions, rather than to \DCPO. However, the motivating theorem~\cite[Theorem
1]{Esik2015} shows that the parametrized fixed point operator in $\CPO_m$
satisfies the same identities as the one found in \CPO{} (\ie, with continuous
rather than monotone functions). Since the parametrized fixed point operator of
\DCPO{} is constructed \emph{precisely} as it is in \CPO{} (noting that
$\omega$-chains are directed sets), this definition is equivalent to the
original.

\subsection{Involutive monoidal categories} 
\label{sub:involutive_monoidal_categories}
An involutive category~\cite{Jacobs2012} is a category in which every object
$X$ can be assigned a \emph{conjugate} object $\inv{X}$ in a functorial way
such that $\iinv{X} \iso X$. A novel idea by Egger~\cite{Egger2008} is to
consider dagger categories as categories enriched in an \emph{involutive
monoidal category}. We will return to this idea in
Section~\ref{sub:the_category_of_continuous_functionals}, and recall the
relevant definitions in the meantime (due to \cite{Jacobs2012}, compare also
with \emph{bar categories}~\cite{Beggs2009}).

\begin{definition}
A category \V{} is \emph{involutive} if it is equipped with a functor $\V
\tot{\inv{(-)}} \V$ (the \emph{involution}) and a natural isomorphism $\id
\Tot{\iota} \iinv{(-)}$ satisfying $\iota_{\inv{X}} = \inv{\iota_X}$.
\end{definition}

Borrowing terminology from linear algebra, we call $\inv{X}$ (respectively
$\inv{f}$) the \emph{conjugate} of an object $X$ (respectively a morphism $f$),
and say that an object $X$ is \emph{self-conjugate} if $X \iso \inv{X}$. Note
that since conjugation is covariant, any category \C{} can be made involutive
by assigning $\inv{X} = X$, $\inv{f} = f$, and letting $\id \Tot{\iota}
\iinv{(-)}$ be the identity in each component; as such, an involution is a
structure rather than a property. Non-trivial examples of involutive categories
include the category of complex vector spaces $\Vect_{\mathbb{C}}$, with the
involution given by the usual conjugation of complex vector spaces; and the
category \Poset{} of partially ordered sets and monotone functions, with the
involution given by order reversal.

When a category is both involutive and (symmetric) monoidal, we say that it is
an \emph{involutive (symmetric) monoidal category} when these two structures
play well together, as in the following definition~\cite{Jacobs2012}.

\begin{definition}
An \emph{involutive (symmetric) monoidal category} is a (symmetric)
monoidal category \V{} which is also involutive, such that the involution is a
monoidal functor, and $\id \To \iinv{(-)}$ is a monoidal natural isomorphism.
\end{definition}

This specifically gives us a natural family of isomorphisms $\inv{X \otimes Y}
\iso \inv{X} \otimes \inv{Y}$, and when the monoidal product is symmetric, this
extends to a natural isomorphism $\inv{X \otimes Y} \iso \inv{Y} \otimes
\inv{X}$. This fact will turn out to be useful later on when we consider dagger
categories as enriched in certain involutive symmetric monoidal categories.

\section{Domain enriched dagger categories} 
\label{sec:dcpo_dagger_categories}
Given a dagger category that also happens to be enriched in domains, we ask how
these two structures ought to interact with one another. Since domain theory
dictates that the well-behaved functions are precisely the continuous ones, a
natural first answer would be to that the dagger should be locally continuous;
however, it turns out that we can make do with less.

\begin{definition}
  Say that a dagger structure on \DCPO-category is \emph{monotone} if the 
  dagger is locally monotone, \ie, if $f \dle g$ implies $f^\dagger \dle
  g^\dagger$ for all $f$ and $g$.
\end{definition}

In the following, we will use the terms ``\DCPO-category with a monotone dagger
structure'' and ``\DCPO-\dag-category'' interchangably. That this is sufficient
to get what we want -- in particular to obtain local continuity of the dagger
-- is shown in the following lemma.

\begin{lemma}
  In any \DCPO-$\dagger$-category, the dagger is an order
  isomorphism on morphisms; in particular it is continuous and strict.
\end{lemma}
\begin{proof}
  For \C{} a dagger category, $\C \iso \C^{\op}$ so $\C(X,Y) \iso \C^{\op}(X,Y) 
  = \C(Y,X)$ for all objects $X,Y$; that this isomorphism of hom-objects is an 
  order isomorphism follows directly by local monotony.\qed
\end{proof}

Let us consider a few examples of \DCPO-\dag-categories.

\begin{example}
  The category \Rel{} of sets and relations is a dagger category, with the
  dagger given by $R^\dag = R^{\circ}$, the relational converse of $R$ (\ie,
  defined by $(y,x) \in R^{\circ}$ iff $(x,y) \in R$) for each such relation. It
  is also enriched in \DCPO{} by the usual subset ordering: Since a relation
  $\mathcal{X} \to \mathcal{Y}$ is nothing more than a subset of $\mathcal{X}
  \times \mathcal{Y}$, equipped with the subset order $- \subseteq -$ we have 
  that $\sup(\Delta) = \bigcup_{R \in \Delta} R$ for any directed set $\Delta
  \subseteq \Rel(\mathcal{X}, \mathcal{Y})$. It is also pointed, with the least 
  element of each homset given by the empty relation.
  
  To see that this is a monotone dagger structure, let $\mathcal{X} \tot{R,S}
  \mathcal{Y}$ be relations and suppose that $R \subseteq S$. Let $(y,x) \in
  R^{\circ}$. 
  Since $(y,x) \in R^{\circ}$ we have $(x,y)
  \in R$ by definition of the relational converse, and by the assumption that
  $R \subseteq S$ we also have $(x,y) \in S$. But then $(y,x) \in S^{\circ}$ by
  definition of the relational converse, 
  so $R^\dag = R^{\circ} \subseteq S^{\circ} = S^\dag$ follows by
  extensionality.
\end{example}

\begin{example}
  We noted earlier that the category \PInj{} of sets and partial injective
  functions is a dagger subcategory of \Rel, with $f^\dag$ given by the partial
  inverse (a special case of the relational converse) of a partial injection
  $f$. Further, it is also a \DCPO-subcategory of \Rel; in \PInj, this becomes
  the relation that for $X \tot{f,g} Y$, $f \dle g$ iff for all $x \in X$, if
  $f$ is defined at $x$ and $f(x) = y$, then $g$ is also defined at $x$ and
  $g(x) = y$. Like \Rel{}, it is pointed with the nowhere defined partial
  function as the least element of each homset. That $\sup(\Delta)$ for some
  directed $\Delta \subseteq \PInj(X,Y)$ is a partial injection follows
  straightforwardly, and that this dagger structure is monotone follows by an
  argument analogous to the one for \Rel.
\end{example}

\begin{example}
  More generally, any \emph{join inverse category} (see \cite{Guo2012}), of 
  which \PInj{} is one, is a \DCPO-\dag-category. Inverse categories are 
  canonically dagger categories enriched in partial orders. That this 
  extends to \DCPO-enrichment in the presence of joins is shown in 
  \cite{Kaarsgaard2017}; that the canonical dagger is 
  monotonous with respect to the partial order is an elementary result (see,
  \eg, \cite[Lemma 2]{Kaarsgaard2017}).
\end{example}

\begin{example}
  The category \DStoch{} of finite sets and \emph{subnormalized doubly 
  stochastic maps} is
  an example of a probabilistic \DCPO-\dag-category. A subnormalized doubly
  stochastic map
  $X \tot{f} Y$, where $|X| = |Y| = n$, is given by an $n \times n$ matrix $A =
  [a_{ij}]$ with non-negative real entries such that $\sum_{i=1}^n a_{ij} \le
  1$ and $\sum_{j = 1}^n a_{ij} \le 1$. Composition is given by the usual
  multiplication of matrices. 
  
  This is a dagger category with the dagger given
  by matrix transposition. It is also enriched in \DCPO{} by ordering 
  subnormalized doubly stochastic maps entry-wise (\ie, $A \le B$ if $a_{ij}
  \le b_{ij}$ for all $i,j$), with the everywhere-zero matrix as the least
  element in each homset, and with suprema of directed sets given by computing
  suprema entry-wise. That this dagger structure is monotone follows by the
  fact that if $A \le B$, so $a_{ij} \le b_{ij}$ for all $i,j$, then also
  $a_{ji} \le b_{ji}$ for all $j,i$, which is precisely to say that $A^\dag =
  A^T \le B^T = B^\dag$.
\end{example}

As such, in terms of computational content, these are examples of deterministic,
nondeterministic, and probabilistic \DCPO-\dag-categories. We will
also discuss the related category $\CPS(\FHilb)$, used to
model quantum phenomena, in Section~\ref{sec:applications}.

\subsection{The category of continuous functionals} 
\label{sub:the_category_of_continuous_functionals} 
We illustrate here the idea of dagger categories as categories enriched in an
involutive monoidal category by an example that will be used throughout the
remainder of this article: Enrichment in a suitable subcategory of \DCPO. It is
worth stressing, however, that the construction is \emph{not} limited to dagger
categories enriched in \DCPO; any dagger category will do. As we will see
later, however, this canonical involution turns out to be very useful when
\DCPO-\dag-categories are considered.

Let \C{} be a \DCPO-\dag-category. We define an induced (full monoidal)
subcategory of \DCPO, call it $\DcpoOp(\C)$, which enriches \C{} (by its
definition) as follows:

\begin{definition}
  For a \DCPO-\dag-category \C, define $\DcpoOp(\C)$ to have as objects all
  objects $\Theta, \Lambda$ of \DCPO{} of the form $\C(X,Y)$, $\C^{\op}(X,Y)$
  (for all objects $X,Y$ of \C), $1$, and $\Theta \times \Lambda$ (with $1$
  initial object of \DCPO, and $- \times -$ the cartesian product), and as
  morphisms all continuous functions between these.
\end{definition}

In other words, $\DcpoOp(\C)$ is the (full) cartesian subcategory of \DCPO{}
generated by objects used in the enrichment of \C, with all continuous maps
between these. That the dagger on \C{} induces an involution on $\DcpoOp(\C)$
is shown in the following theorem.

\begin{theorem}
  $\DcpoOp(\C)$ is an involutive symmetric monoidal category.
\end{theorem}
\begin{proof}
  On objects, define an involution $\inv{(-)}$ with respect to the cartesian 
  (specifically symmetric monoidal) product of \DCPO{} as follows, for all 
  objects $\Theta,\Lambda,\Sigma$ of $\DcpoOp(\C)$: $\inv{\C(X,Y)} = 
  \C^{\op}(X,Y)$,
  $\inv{\C^{\op}(X,Y)} = \C(X,Y)$, $\inv{1} = 1$, and $\inv{\Theta \times
  \Lambda} = \inv{\Theta} \times \inv{\Lambda}$.
  To see that this is well-defined, recall that $\C \iso \C^{\op}$ for any
  dagger category $\C$, so in particular there is an isomorphism witnessing
  $\C(X,Y) \iso \C^{\op}(X,Y)$ given by the mapping $f \mapsto f^\dag$. But
  then $\C^{\op}(X,Y) = \{f^\dag \mid f \in \C(X,Y)\}$, so if $\C(X,Y) =
  \C(X',Y')$ then $\inv{\C(X,Y)} = \C^{\op}(X,Y) = \{f^\dag \mid f \in
  \C(X,Y)\} = \{f^\dag \mid f \in \C(X',Y')\} = \C^{\op}(X',Y') =
  \inv{\C(X',Y')}$. That $\inv{\C^{\op}(X,Y)} = \C(X,Y)$ is well-defined 
  follows by analogous argument.
  
  On morphisms, we define a
  family $\xi$ of isomorphisms by $\xi_I = \id_I$, $\xi_{\C(X,Y)} =
  (-)^\dagger$, $\xi_{\C^{\op}(X,Y)} = (-)^\dagger$, and $\xi_{\Theta \times
  \Lambda} = \xi_\Theta \times \xi_\Lambda$, and then define
  $$\inv{\Theta \tot{\phi} \Lambda} = \inv{\Theta} \tot{\xi_\Theta^{-1}} \Theta 
  \tot{\phi} \Lambda \tot{\xi_\Lambda} 
  \inv{\Lambda}.$$
  This is functorial as $\inv{\id_\Theta} = \xi_\Theta \circ \id_\Theta \circ \xi_\Theta^{-1} = 
  \xi_\Theta \circ \xi_\Theta^{-1} = \id_{\inv{\Theta}}$, and for $\Theta \tot{\phi} \Lambda \tot{\psi} \Sigma$, 
  $$\inv{\psi \circ \phi} = \xi_\Sigma \circ \psi \circ \phi \circ 
  \xi_\Theta^{-1} = \xi_\Sigma \circ \psi \circ \xi_\Lambda^{-1} \circ \xi_\Lambda \circ \phi \circ \xi_\Theta^{-1} = \inv{\psi} \circ 
  \inv{\phi}.$$ 
  Finally, since the involution is straightforwardly a monoidal functor, and
  since the natural transformation $\id \To \iinv{(-)}$ can be chosen to be the
  identity since all objects of $\DcpoOp(\C)$ satisfy $\iinv{\Theta} = \Theta$ 
  by definition, this is an involutive symmetric monoidal category.\qed
\end{proof}

The resulting category $\DcpoOp(\C)$ can very naturally be thought of as the
induced \emph{category of (continuous) functionals} (or second-order functions)
of \C.

Notice that this is a special case of a more general construction on dagger
categories: For a dagger category \C{} enriched in some category \V{} (which
could simply be \Set{} in the unenriched case), one can construct the category
$\V\mathbf{Op}(\C)$, given on objects by the image of the hom-functor $\C(-,-)$
closed under monoidal products, and on morphisms by all morphisms of \V{}
between objects of this form. Defining the involution as above, $\V\mathbf{Op}(\C)$ can be shown to be involutive monoidal. 

\begin{example}
  One may question how natural (in a non-technical sense) the choice of
  involution on $\DcpoOp(\C)$ is. One instance where it turns out to be useful 
  is in the context of dagger adjunctions (see \cite{HeunenKarvonen2016} for
  details), that is, adjunctions between dagger categories where both functors
  are dagger functors. 
  
  Dagger adjunctions have no specified left and right adjoint, as all such
  adjunctions can be shown to be ambidextrous in the following way: Given $F
  \dashv G$ between endofunctors on \C, there is a natural isomorphism
  $\C(FX, Y) \tot{\alpha_{X,Y}} \C(X, GY)$. Since \C{} is a dagger category, we
  can define a natural isomorphism $\C(X, FY) \tot{\beta_{X,Y}} \C(GX, Y)$ by
  $f \mapsto \alpha_{Y,X}(f^\dag)^\dag$, \ie, by the composition
  $$
    \C(X, FY) \tot{\xi} \C(FY, X) \tot{\alpha_{Y,X}} \C(Y, GX) \tot{\xi} 
    \C(GX, Y)
  $$
  which then witnesses $G \dashv F$ (as it is a composition of natural
  isomorphisms). But then $\beta_{X,Y}$ is defined precisely to be
  $\inv{\alpha_{Y,X}}$ when $F$ and $G$ are endofunctors.
\end{example}

\subsection{Daggers and fixed points} 
\label{sub:daggers_and_fixed_points}
In this section we consider the morphisms of $\DcpoOp(\C)$ in some detail, for
a \DCPO-\dag-category \C. Since least fixed points of morphisms are such a
prominent and useful feature of \DCPO-enriched categories, we ask how these
behave with respect to the dagger. To answer this question, we transplant the notion of a \emph{fixed point
adjoint} from \cite{Kaarsgaard2017} to \DCPO-\dag-categories, where
an answer to this question in relation to the more specific \emph{join inverse
categories} was given: 
\begin{definition}
A functional $\C(Y,X) \tot{\phi_\ddagger} \C(Y,X)$ is \emph{fixed point
adjoint} to a functional $\C(X,Y) \tot{\phi} \C(X,Y)$ iff $(\fix \phi)^\dag =
\fix \phi_\ddagger$.
\end{definition}
Note that this is symmetric: If $\phi_\ddagger$ is fixed point adjoint to
$\phi$ then $\fix (\phi_\ddagger)^\dag = (\fix \phi)^{\dag\dag} = \fix \phi$,
so $\phi$ is also fixed point adjoint to $\phi_\ddagger$. As shown in the
following theorem, it turns out that the conjugate $\inv{\phi}$ of a functional
$\phi$ is precisely fixed point adjoint to it. This is a generalization of a
theorem from \cite{Kaarsgaard2017}, where a more ad-hoc formulation was shown for join inverse categories, which constitute a non-trivial subclass of \DCPO-\dag-categories.

\begin{theorem}
  Every functional is fixed point adjoint to its conjugate.
\end{theorem}
\begin{proof}
  The proof applies the exact same construction as in \cite{Kaarsgaard2017},
  since being a \DCPO-$\dagger$-category suffices, and the 
  constructed fixed point adjoint turns out to be the exact same. Let $\C(X,Y) 
  \tot{\phi} \C(X,Y)$ be a functional. Since
  $\inv{\phi} = \xi_{\C(X,Y)} \circ \phi \circ \xi_{\C(X,Y)}^{-1}$,
  $$\bar{\phi}^n = \left(\xi_{\C(X,Y)} \circ \phi \circ
  \xi_{\C(X,Y)}^{-1}\right)^n = \xi_{\C(X,Y)} \circ \phi^n \circ
  \xi_{\C(X,Y)}^{-1}$$ and so
  \begin{align*}
    \fix \inv{\phi} & = \sup \{\bar{\phi}^n(\bot_{Y,X})\}_{n \in \omega}
    = \sup \{\phi^n(\bot_{Y,X}^\dagger)^\dagger\} = \sup 
    \{\phi^n(\bot_{X,Y})^\dagger\} \\ & = \sup 
    \{\phi^n(\bot_{X,Y})\}^\dagger = (\fix \phi)^\dagger
  \end{align*}
  as desired.\qed
\end{proof}

This theorem is somewhat surprising, as the conjugate came out of the
involutive monoidal structure on $\DcpoOp(\C)$, which is not specifically
related to the presence of fixed points. As previously noted, had \C{} been
enriched in another category \V, we would still be able to construct a category
$\V\mathbf{Op}(\C)$ of \V-functionals with the \emph{exact same} involutive
structure.

As regards recursion, this theorem underlines the slogan that
\emph{reversibility is a local phenomenon}: To construct the inverse
to a recursively defined morphism $\fix \phi$, it suffices to invert the local morphism $\phi$ at each step (which is essentially what is done by the conjugate $\inv{\phi}$) in order to construct the global inverse $(\fix \phi)^\dag$.

Parametrized functionals and their external fixed points are also interesting
to consider in this setting, as some examples of \DCPO-\dag-categories (\eg,
\PInj) fail to have an internal hom. For example, in a dagger category with
objects $L(X)$ corresponding to ``lists of $X$'' (usually constructed as the
fixed point of a suitable functor), one could very reasonably construe the
usual map-function not as a higher-order function, but as a family of morphisms
$LX \tot{\text{map}\langle f \rangle} LY$ indexed by $X \tot{f} Y$ -- or, in
other words, as a functional $\C(X,Y) \tot{\text{map}} \C(LX,LY)$. Indeed, this
is how certain higher-order behaviours are mimicked in the reversible
functional programming language Theseus (see also
Section~\ref{sec:applications}).

To achieve such parametrized fixed points of functionals, we naturally need a
parametrized fixed point operator on $\DcpoOp(\C)$ satisfying the appropriate
equations -- or, in other words, we need $\DcpoOp(\C)$ to be an \emph{iteration
category}. That $\DcpoOp(\C)$ is such an iteration category follows immediately
by its definition (\ie, since $\DcpoOp(\C)$ is a full subcategory of \DCPO, we
can define a parametrized fixed point operator in $\DcpoOp(\C)$ to be precisely 
the one in \DCPO), noting that parametrized fixed points preserve continuity.

\begin{lemma}
  $\DcpoOp(\C)$ is an iteration category.
\end{lemma}

For functionals of the form $\C(X,Y) \times \C(P,Q) \tot{\psi} \C(X,Y)$, we can
make a similar definition of a \emph{parametrized fixed point adjoint}:

\begin{definition}
  A functional $\C(X,Y) \times \C(P,Q) \tot{\psi_\ddagger} \C(X,Y)$ is
  \emph{parametrized fixed point adjoint} to a functional $\C(X,Y) \times
  \C(P,Q) \tot{\psi} \C(X,Y)$ iff $(\pfix \psi)(p)^\dagger = (\pfix
  \psi_\ddagger)(p^\dag)$.
\end{definition}

We can now show a similar theorem for parametrized fixed points of functionals
and their conjugates:

\begin{theorem}\label{thm:pfix_rev}
  Every functional is parametrized fixed point adjoint to its conjugate.
\end{theorem}
\begin{proof}
Let $\C(X,Y) \times \C(P,Q) \tot{\psi} \C(X,Y)$ be a functional.
We start by showing that $\bar{\psi}^n(f,p) =
\psi^n(f^\dagger,p^\dagger)^\dagger$ for all $Y \tot{f} X$, $Q \tot{p} P$, and
$n \in \mathbb{N}$, by induction on $n$. For $n=0$ we have $$\bar{\psi}^0(f,p)
= f = f^{\dagger\dagger} =(f^\dagger)^\dagger =
\psi^0(f^\dagger,p^\dagger)^\dagger.$$ Assuming now the induction hypothesis
for some $n$, we have
\begin{align*}
\bar{\psi}^{n+1}(f,p) & = \bar{\psi}(\bar{\psi}^n(f,p),p) = 
\bar{\psi}(\psi^n(f^\dag,p^\dag)^\dag,p) = 
\psi(\psi^n(f^\dag,p^\dag)^{\dag\dag},p^\dag)^\dag \\
& = \psi(\psi^n(f^\dag,p^\dag),p^\dag)^\dag = \psi^{n+1}(f^\dag,p^\dag)^\dag
\end{align*}
Using this fact, we now get
\begin{align*} 
(\pfix \inv{\psi})(p^\dagger) & = \sup_{n \in \omega}
\{\bar{\psi}^n(\bot_{Y,X},p^\dagger)\} = \sup_{n \in \omega}
\{\psi^n(\bot_{Y,X}^\dagger,p^{\dag\dag})^\dagger\} \\ & = \sup_{n \in 
\omega} \{\psi^n(\bot_{X,Y},p)\}^\dagger = (\pfix \psi)(p)^\dagger
\end{align*}
which was what we wanted. \qed
\end{proof}

Again, this theorem highlights the local nature of reversibility, here in the
presence of additional parameters. We observe further the following highly
useful property of parametrized fixed points in $\DcpoOp(\C)$:

\begin{lemma}
Parametrized fixed points in $\DcpoOp(\C)$ preserve conjugation.
\end{lemma}
\begin{proof}
Let $\C(X,Y) \times \C(P,Q) \tot{\psi} \C(X,Y)$ be continuous, and $P \tot{p}
Q$. Then $\inv{\pfix \psi}(p) = (\xi \circ (\pfix \psi) \circ \xi^{-1})(p) =
(\pfix \psi)(p^\dagger)^\dagger = (\pfix \inv{\psi})(p)^{\dag\dag} = (\pfix
\inv{\psi})(p)$, so $\inv{\pfix \psi} = \pfix \inv{\psi}$.\qed
\end{proof}

Note that a lemma of this form only makes sense for parametrized fixed points,
as the usual fixed point of a functional $\C(X,Y) \tot{\phi} \C(X,Y)$ results
in a morphism $X \tot{\fix \phi} Y$ in \C{}, not a functional in $\DcpoOp(\C)$.

\subsection{Naturality and self-conjugacy} 
\label{sub:naturality}
We now consider the behaviour of functionals and their parametrized fixed
points when they are natural. For example, given a natural family of
functionals $\C(FX,FY) \tot{\alpha_{X,Y}} \C(GX,GY)$ natural in $X$ and $Y$
(for dagger endofunctors $F$ and $G$ on \C), what does it mean for such a
family to be well-behaved with respect to the dagger on \C? We would certainly
want that such a family preserves the dagger, in the sense that
$\alpha_{X,Y}(f)^\dag = \alpha_{Y,X}(f^\dag)$ in each component $X,Y$. It turns
out that this, too, can be expressed in terms of conjugation of functionals.

\begin{lemma}\label{lem:self_conjugate}
  Let $\C(FX,FY) \tot{\alpha_{X,Y}} \C(GX,GY)$ be a family of functionals
  natural in $X$ and $Y$. Then $\alpha_{X,Y}(f)^\dag = \alpha_{Y,X}(f^\dag)$
  for all $X \tot{f} Y$ iff $\alpha_{X,Y} = \inv{\alpha_{Y,X}}$.
\end{lemma}
\begin{proof}
  Suppose $\alpha_{X,Y}(f)^\dag = \alpha_{Y,X}(f^\dag)$. Then $\alpha_{X,Y}(f) 
  = \alpha_{X,Y}(f)^{\dag\dag} = \alpha_{Y,X}(f^\dag)^\dag = 
  \inv{\alpha_{Y,X}}(f)$, so $\alpha_{X,Y} = \inv{\alpha_{Y,X}}$. Conversely, 
  assuming $\alpha_{X,Y} = \inv{\alpha_{Y,X}}$ we then have for all $X \tot{f} 
  Y$ that $\alpha_{X,Y}(f) = \alpha_{Y,X}(f^\dag)^\dag$, so 
  $\alpha_{X,Y}(f)^\dag = \alpha_{Y,X}(f^\dag)^{\dag\dag} = 
  \alpha_{Y,X}(f^\dag)$. \qed
\end{proof}

If a natural transformation $\alpha$ satisfies $\alpha_{X,Y} =
\inv{\alpha_{Y,X}}$ in all components $X,Y$, we say that it is
\emph{self-conjugate}. An important example of a self-conjugate natural
transformation is the \emph{dagger trace operator}, as detailed in the
following example.

\begin{example}
  A trace operator~\cite{JoyalStreetVerity1996} on a braided monoidal category
  \D{} is family of functionals $$\D(X \otimes U, Y \otimes U) \tot{\Tr_{X,Y}^U}
  \D(X,Y)$$ subject to certain equations (naturality in $X$ and $Y$,
  dinaturality in $U$, etc.). Traces have been used to model features from
  partial traces in tensorial vector spaces~\cite{HasegawaHofmannPlotkin2008} 
  to tail recursion in programming
  languages~\cite{Abramsky1996,BentonHyland2003,Hasegawa1997}, and occur
  naturally in tortile monoidal categories~\cite{JoyalStreetVerity1996} and
  unique decomposition categories~\cite{Haghverdi2000,Hoshino2012}.
  
  A \emph{dagger trace operator} on a dagger category (see, \eg,
  \cite{Selinger2011}) is precisely a trace operator on a dagger monoidal
  category (\ie, a monoidal category where the monoidal functor is a dagger
  functor) that satisfies $\Tr_{X,Y}^U(f)^\dag = \Tr_{Y,X}^U(f^\dag)$ in all
  components $X,Y$. Such traces have been used to model reversible tail
  recursion in reversible programming
  languages~\cite{James2012,James2014,Kaarsgaard2017}, and also occur in the
  \emph{dagger compact closed categories} (see, \eg, \cite{Selinger2012}) used
  to model quantum theory. In light of Lemma~\ref{lem:self_conjugate},
  dagger traces are important examples of self-conjugate natural
  transformations on dagger categories.
\end{example}

Given the connections between (di)naturality and parametric
polymorphism~\cite{Wadler1989,Bainbridge1990}, one would wish that parametrized
fixed points preserve naturality. Luckily, this does turn out to be the case,
as shown in the proof of the following theorem.

\begin{theorem}\label{thm:nat}
  If $\C(FX,FY) \times \C(GX,GY) \tot{\alpha_{X,Y}} \C(FX,FY)$ is natural in 
  $X$ and $Y$, so is its parametrized fixed point.
\end{theorem}
\begin{proof}
  See appendix.\qed
\end{proof}
This theorem can be read as stating that, just like reversibility, a recursive
polymorphic map can be obtained from one that is only locally polymorphic. 
Combining this result with Lemma~\ref{lem:self_conjugate} regarding
self-conjugacy, we obtain the following corollary.
\begin{corollary}
If $\C(FX,FY) \times \C(GX,GY) \tot{\alpha_{X,Y}} \C(FX,FY)$ is a
self-conjugate natural transformation, so is $\pfix \alpha_{X,Y}$.
\end{corollary}
\begin{proof}
  If $\alpha_{X,Y} = \inv{\alpha_{Y,X}}$ for all $X,Y$ then also
  $\pfix \alpha_{X,Y} = \pfix \inv{\alpha_{Y,X}}$, which is further natural in 
  $X$ and $Y$ by Theorem~\ref{thm:nat}. But then $\inv{\pfix \alpha_{X,Y}} = 
  \pfix \inv{\alpha_{X,Y}} = \pfix \alpha_{Y,X}$, as parametrized fixed points 
  preserve conjugation. \qed
\end{proof}

\section{Applications and future work} 
\label{sec:applications}

\paragraph*{Reversible programming languages} 
\label{par:reversible_programming_languages}
Theseus~\cite{James2014} is a typed reversible functional programming language
similar in syntax and spirit to Haskell. It has support for recursive data
types, as well as reversible tail recursion using so-called \emph{typed
iteration labels} as syntactic sugar for a dagger trace operator. Theseus is
based on the $\Pi$-family of reversible combinator calculi~\cite{James2012},
which bases itself on dagger traced symmetric monoidal categories augmented
with a certain class of algebraically $\omega$-compact functors. 

Theseus also supports \emph{parametrized functions}, that is, families of
reversible functions indexed by reversible functions of a given type, with the
proviso that parameters must be passed to parametrized maps statically. For
example, (if one extended Theseus with polymorphism) the reversible map
function would have the signature $\mathit{map} :: (a \leftrightarrow b) \to
([a] \leftrightarrow [b])$, and so $\mathit{map}$ is not in itself a reversible
function, though $\mathit{map}\ \langle f \rangle$ is (for some suitable
function $f$ passed statically). This gives many of the benefits of
higher-order programming, but without the headaches of higher-order reversible
programming.

The presented results show very directly that we can extend Theseus with a
fixed point operator for general recursion while maintaining desirable
inversion properties, rather than making do with the simpler tail recursion.
Additionally, the focus on the continuous functionals of \C{} given by the
category $\DcpoOp(\C)$ also highlights the feature of parametrized functions in
Theseus, and our results go further to show that even parametrized functions
that use general recursion not only have desirable inversion properties,
but also preserve naturality, the latter of which is useful for extending Theseus with parametric polymorphism.

\paragraph*{Quantum programming languages} 
\label{par:quantum_computing}
An interesting possibility as regards quantum programming languages is the
category $\CPS(\FHilb)$ (see \cite{Coecke2016} for details on the \CPS-construction), which is dagger compact closed and equivalent to the
category of finite-dimensional $C^*$-algebras and completely positive
maps~\cite{Coecke2016}. Since finite-dimensional $C^*$-algebras are
specifically von Neumann algebras, it follows (see
\cite{Cho2014,Rennela2014}) that this category is enriched in the category of
\emph{bounded} directed complete partial orders; and since it inherits the
dagger from \FHilb{} (and is locally ordered by the pointwise extension of the 
Löwner order restricted to positive operators), the dagger structure is
monotone, too. As such, the presented results ought to apply in this case as 
well -- modulo concerns of boundedness -- though this warrants more careful
study.

\paragraph*{Dagger traces in \DCPO-\dag-categories}
Given a suitable monoidal tensor (\eg, one with the zero object as tensor unit)
and a partial additive structure on morphisms, giving the category the structure of a \emph{unique decomposition category}~\cite{Haghverdi2000,Hoshino2012}, a trace operator can be canonically constructed. In previous work~\cite{Kaarsgaard2017}, the author (among others)
demonstrated that a certain class of \DCPO-\dag-categories, namely join inverse
categories, had a dagger trace under suitably mild assumptions. It is
conjectured that this theorem may be generalized to other \DCPO-\dag-categories
that are not necessarily inverse categories, again provided that certain
assumptions are satisfied.

\paragraph*{Involutive iteration categories} 
\label{par:involutive_iteration_theories}
As it turned out that the category $\DcpoOp(\C)$ of continuous functionals on
\C{} was both involutive and an iteration category, an immediate question to
ask is how the involution functor ought to interact with parametrized fixed 
points in the general case. A remarkable fact of iteration categories is that
they are defined to be cartesian categories that satisfy all equations of
parametrized fixed points that hold in the category $\mathbf{CPO}_m$ of
$\omega$-complete partial orders and \emph{monotone} functions, yet also have a
complete (though infinite) equational axiomatization~\cite{Esik2015}.

We have provided an example of an interaction between
parametrized fixed points and the involution functor here, namely that
$\DcpoOp(\C)$ satisfies $\inv{\pfix \psi} = \pfix \inv{\psi}$. It could be
interesting to search for examples of involutive iteration categories in the
wild (as candidates for a semantic definition), and to see if Ésik's
axiomatization could be extended to accomodate for the involution functor in
the semantic category.


\section{Conclusion and related work} 
\label{sec:conclusion} 
We have developed a notion of \DCPO-categories with a monotone dagger structure
(of which \PInj{}, \Rel{}, and \DStoch{} are examples, and $\CPS(\FHilb)$ is
closely related), and shown that these categories can be taken to be enriched
in an induced involutive monoidal category of continuous functionals. With
this, we were able to account for (ordinary and parametrized) fixed point
adjoints as arising from conjugation of the functional in the induced
involutive monoidal category, to show that parametrized fixed points preserve
conjugation and naturality, and that natural transformations that preserve the
dagger are precisely those that are self-conjugate. We also described a number
of potential applications in connection with reversible and quantum computing.

A great deal of work has been carried out in recent years on the domain theory
of quantum computing, with noteworthy results in categories of von Neumann
algebras (see, \eg, \cite{Rennela2014,Cho2014,Jacobs2015,Cho2015}). Though the
interaction between dagger structure and the domain structure on homsets was
not the object of study, Heunen considers the similarities and differences of
\FHilb{} and \PInj, also in relation to domain structure on homsets, in
\cite{Heunen2013}, though he also notes that \FHilb{} fails to enrich in
domains as composition is not even monotone (this is not to say that domain
theory and quantum computing do not mix; only that \FHilb{} is the wrong
category to consider for this purpose). Finally, dagger traced symmetric
monoidal categories, with the dagger trace serving as an operator for
reversible tail recursion, have been studied in connection with reversible
combinator calculi~\cite{James2012} and functional programming~\cite{James2014}.

\bibliographystyle{splncs04}
\bibliography{library}

\appendix

\newpage

\section{Omitted proofs} 
\label{sec:omitted_proofs}
\subsection{Proof of Theorem~\ref{thm:nat}} 
\label{sub:proof_of_theorem_ref_thm_nat}
Suppose that $\alpha$ is natural in $X$ and $Y$, \ie, the following diagram 
commutes for all $X,Y$.
\begin{center}
\begin{tikzpicture}
\node (CFXFYGXGY) {$\C(FX,FY) \times \C(GX,GY)$};
\node[right=20mm of CFXFYGXGY] (CFXFY) {$\C(FX,FY)$};
\node[below of=CFXFYGXGY] (CFX'FY'GX'GY') {$\C(FX',FY') \times \C(GX',GY')$};
\node[below of=CFXFY] (CFX'FY') {$\C(FX',FY')$};

\draw[->] (CFXFYGXGY) to node {\scriptsize $\alpha_{X,Y}$} (CFXFY);
\draw[->] (CFXFYGXGY) to node [swap] {\scriptsize $Ff \times Gf \circ - \circ
Fg \times Gg$} (CFX'FY'GX'GY');
\draw[->] (CFXFY) to node {\scriptsize $Ff \circ - \circ Fg$} (CFX'FY');
\draw[->] (CFX'FY'GX'GY') to node [swap] {\scriptsize $\alpha_{X',Y'}$}
(CFX'FY');
\end{tikzpicture}
\end{center}
Under this assumption, we start by showing naturality of $\alpha^n$ for all
$n \in \mathbb{N}$, \ie{}, for all $GX \tot{p} GY$
\begin{equation*}
  \alpha^n_{X',Y'}(\bot_{X',Y'}, Gf \circ p \circ Gg) = 
  Ff \circ \alpha^n_{X,Y}(\bot_{X,Y},p) \circ Fg
\end{equation*}
by induction on $n$. For $n = 0$ we have 
\begin{align*}
  \alpha^0_{X',Y'}(\bot_{X,Y}, Gf \circ p \circ Gg) & = \bot_{X',Y'} \\
  & = Ff \circ \bot_{X,Y} \circ Fg \\
  & = Ff \circ \alpha_{X,Y}^0(\bot_{X,Y},p) \circ Fg.
\end{align*}
where $Ff \circ \bot_{X,Y} \circ Fg = \bot_{X',Y'}$ by strictness of 
composition.
Assuming the induction hypothesis now for some $n$, we have
\begin{align*}
  \alpha^{n+1}_{X',Y'}(\bot_{X',Y'}, Gf \circ p \circ Gg) & =
  \alpha_{X',Y'}(\alpha^{n}_{X',Y'}(\bot_{X',Y'}, Gf \circ p \circ Gg), Gf 
  \circ p \circ Gg) \\
  & = \alpha_{X',Y'}(Ff \circ \alpha^n_{X,Y}(\bot_{X,Y},p) \circ Fg, 
  Gf \circ p \circ Gg) \\
  & = Ff \circ \alpha_{X,Y}(\alpha^n_{X,Y}(\bot_{X,Y},p), p) \circ Fg \\
  & = Ff \circ \alpha^{n+1}_{X,Y}(\bot_{X,Y},p) \circ Fg
\end{align*}
so $\alpha^n$ is, indeed, natural for any choice of $n \in \mathbb{N}$. But 
then
\begin{align*}
  (\pfix \alpha_{X',Y'})(Gf \circ p \circ Gg) & =
  \sup_{n \in \omega} \left\{\alpha^n_{X',Y'}(\bot_{X',Y'},Gf \circ p \circ 
  Gg)\right\} \\ 
  & = \sup_{n \in \omega} \left\{\alpha^n_{X',Y'}(Ff \circ \bot_{X,Y} 
  \circ Fg,Gf \circ p 
  \circ Gg)\right\} \\
  & = \sup_{n \in \omega} \left\{Ff \circ \alpha^n_{X,Y}(\bot_{X,Y},p) 
  \circ Fg \right\} \\
  & = Ff \circ \sup_{n \in \omega} \left\{\alpha^n_{X,Y}(\bot_{X,Y},p)
  \right\} \circ Fg \\
  & = Ff \circ (\pfix \alpha_{X,Y})(p) \circ Fg
\end{align*}
so $\pfix \alpha_{X,Y}$ is natural as well.\qed

\end{document}